\documentclass[11pt]{amsart}
\usepackage{amssymb,amsmath,amsthm}
\usepackage{bm}

\input{amssym.def}
\newcommand{\qfor}{~\text{for}~}
\newcommand{\qand}{~\text{and}~}

\newcommand{\bt}{\begin{Theorem}}
\newcommand{\et}{\end{Theorem}}
\newcommand{\bi}{\begin{itemize}}
\newcommand{\ei}{\end{itemize}}
\newcommand{\bea}{\begin{eqnarray}}
\newcommand{\eea}{\end{eqnarray}}

\theoremstyle{plain}
\newtheorem{Theorem}{\sc Theorem}
\newtheorem{Lemma}{\sc Lemma}
\newtheorem{Proposition}{\sc Proposition}
\newtheorem{Corollary}{\sc Corollary}

\theoremstyle{definition}

\theoremstyle{remark}
\newtheorem{Remark}{\sc Remark}

\def\u{\mbox{{$\mathbf{u}$}}}%
\def\v{\mbox{{$\mathbf{v}$}}}%
\def\z{\mbox{{$\mathbf{z}$}}}%
\def\w{\mbox{{$\mathbf{w}$}}}%

\newcommand{\be}{\begin{equation}}
\newcommand{\ee}{\end{equation}}

\def\ops{{operators~}}%
\def\hs{{Hilbert space~}}%
\def\hsp{Hilbert space}%

\newcommand{\rk}{reproducing kernel }%

\newcommand{\bb}[1]{\mathbb #1}
\newcommand{\cl}[1]{\mathcal #1}

\def\C{\mbox{${\mathbb C}$}}
\def\B{\mbox{${\mathbb B}$}}
\def\D{\mbox{${\mathbb D}$}}

\def\N{\mbox{${\mathbb N}$}}

\newcommand{\alg}[1]{\mbox{${\mathcal A}(#1)$}}%
\def\v{\mbox{{\boldmath $v$}}}%
\newcommand{\kf}[1]{{K(\cdot,#1)}}%

\newcommand{\inner}[2]{\big \langle #1,#2\big \rangle }%

\newcommand{\cla}{\mathcal{A}}

\newcommand{\cle}{\mathcal{E}}

\newcommand{\clh}{\mathcal{H}}

\newcommand{\cll}{\mathcal{L}}
\newcommand{\clm}{\mathcal{M}}
\newcommand{\cln}{\mathcal{N}}
\newcommand{\clo}{\mathcal{O}}

\newcommand{\clr}{\mathcal{R}}
\newcommand{\cls}{\mathcal{S}}
\newcommand{\clt}{\mathcal{T}}

\newcommand{\raro}{\rightarrow}

\def \qed {\hfill \vrule height6pt width 6pt depth 0pt}
\def\textmatrix#1&#2\\#3&#4\\{\bigl({#1 \atop #3}\ {#2 \atop #4}\bigr)}
\def\dispmatrix#1&#2\\#3&#4\\{\left({#1 \atop #3}\ {#2 \atop #4}\right)}

\newcommand{\ben}{\begin{eqnarray*}}
\newcommand{\een}{\end{eqnarray*}}

\newcommand{\NI}{\noindent}

\def\5{{5\superprime}}

\begin{document}
\title[Contractive Hilbert modules and their dilations]{Contractive Hilbert modules and their dilations}

\author[R. G. Douglas]{Ronald G. Douglas}
\address[R. G. Douglas and J. Sarkar]{Texas A\&M University, College Station, Texas 77843}
\email[R. G.  Douglas]{rgd@math.tamu.edu} \email[J.  Sarkar]{jsarkar@math.tamu.edu}
\author[G. Misra]{Gadadhar Misra}
\address[G. Misra]{Indian Institute of Science
Bangalore 560 012}
\email[G. Misra]{gm@math.iisc.ernet.in}
\author[J. Sarkar]{Jaydeb Sarkar}

\address[Present Address of J. Sarkar]{Department of Mathematics, The University of Texas at San Antonio, San Antonio, TX 78249, USA}
\email{jaydeb.sarkar@utsa.edu}

\thanks{The work of RGD and JS was partially supported by a grant from the National Science Foundation, US, while that of GM was supported by Department of Science and Technology, India. This work was carried out during research visits to Texas A\&M and IISc., Bangalore.}
%M\"{o}bius group, ${\rm SU}(1,1)$, holomorphic induction}
%\address{Texas A\&M University, College Station, Texas 77843}
%\today

\begin{abstract}
In this note, we show that a quasi-free Hilbert module $\mathcal R$ defined over the polydisk algebra with kernel function $k(\z, \w)$ admits a unique minimal dilation (actually an isometric co-extension) to the Hardy module over the polydisk if and only if $S^{-1}(\z, \w) k(\z, \w)$ is a positive kernel function, where $S(\z, \w)$ is the Szeg\"{o} kernel for the polydisk. Moreover, we establish the equivalence of such a factorization of the kernel function and a positivity condition, defined using the hereditary functional calculus, which was introduced earlier by Athavale \cite{Ath} and Ambrozie,  Englis and M\"{u}ller \cite{AEM}. An explicit realization of the dilation space is given along with the isometric embedding of the module $\mathcal R$ in it. The proof works for a wider class of Hilbert modules in which the Hardy module is replaced by more general quasi-free Hilbert modules such as the classical spaces on the polydisk or the unit ball in $\mathbb{C}^m$. Some consequences of this more general result are then explored in the case of several natural function algebras.
\end{abstract}
\maketitle

\section*{Introduction}
One of the most far-reaching results in operator theory is the fact that every contraction operator has an essentially unique minimal unitary dilation and a closely related isometric co-extension on which the model theory of Sz.-Nagy and Foias \cite{Na-Fo} is based. This model provides not only a theoretical understanding of the structure of contractions but provides a useful and effective method for calculation.

A key reason this model theory is so incisive is the relatively simple structure of isometries due to von Neumann \cite{vN} .  In particular, every isometry is the direct sum of a unitary and a unilateral shift operator defined on a vector-valued Hardy space.  And, if one makes a modest assumption about the behavior of the powers of the adjoint of the contraction, then the unitary is absent and the isometry involved in the model  is the vector-valued unilateral shift defined on a vector-valued Hardy space on the unit disk.

If one attempts to extend this theory to commuting $m$-tuples of contractions on a Hilbert space, then one quickly runs into trouble. This is particularly true if $m > 2$ in which case the example of Parrott \cite{P} shows that a unitary dilation and hence an isometric co-extension need not exist.  For the $m = 2$ case, Ando's Theorem \cite{A} seems to hold out hope for a model theory since a pair of commuting contractions is known to have a unitary dilation and hence an isometric co-extension. However, such dilations are not necessarily unique and, more critically, the structure of the pair of commuting isometries is not simple.  In particular, the dilation space need not be related to the Hardy space on the bi-disk (see \cite{BDF}).

In this note, we study the question of which commuting pairs (and $m$-tuples) of contractions have an isometric co-extension to the forward shift operators on the Hardy space for the bi-disk (or polydisk). We  take up the question  for commuting contractions on a reproducing kernel Hilbert space which are defined by multiplication by the coordinate functions.

We approach the issue in greater generality - namely, in the context of Hilbert modules over the algebra of polynomials $\mathbb{C}[\z]$ in $m$ variables. Our main tool is to establish a close relationship between the kernel functions for the Hilbert modules in an exact sequence using localization. In particular, using this relation we seek to determine which quasi-free Hilbert modules can be obtained as a quotient module of a fixed Hilbert module of the form $\clr \otimes \cle$ for a "model" quasi-free Hilbert module $\clr$ of multiplicity one and coefficient Hilbert space $\cle$.  Our results for $\clr$ the Hardy module hold not just for the case $m = 2$ but for all $m$.  Of course, the conditions we impose are more restrictive than simply the assumption that the coordinate multipliers are contractive. Our characterization involves the relationship between the two kernel functions and provides an explicit construction of the dilation space using a factorization of the kernel function and a ''reduced'' tensor product as the main tools.

Our main result relates the existence of an $\clr \otimes \cle$ isometric co-extension to the positivity of the kernel function into which the coordinate operators are substituted as well as a factorization criteria for the kernel function itself. The equivalence of the first two conditions was established earlier by Athavale \cite{Ath} (see also \cite{AE}). However, our proof is quite different from his. A key step in our approach involves the hereditary functional calculus of Agler \cite{Agl} which has been an effective tool in constructing analytic models (cf. \cite{AEM}).

We begin by recalling the notion of a quasi-free Hilbert module which is a contractive reproducing kernel Hilbert space. Our main result for the vector-valued Hardy module is Theorem \ref{cor1} which determines when a large class of contractive quasi-free Hilbert modules over the polydisk algebra $\mathcal A(\mathbb D^m)$ admits a dilation to the $\mathcal E$ - valued Hardy module $H^2(\mathbb D^m)\otimes \mathcal E$ for $m \geq 1$ and some Hilbert space $\mathcal E$.  We provide an example showing that a contractive quasi-free Hilbert module over the bi-disk algebra $\mathcal A(\mathbb D^2)$ need not admit an isometric co-extension to the Hardy module $H^2(\mathbb D^2)$. We also consider corollaries which provide analogous results for a class of Hilbert modules which includes the Bergman module over the unit ball or polydisk. In the next section, we consider the existence of spherical Drury-Arveson shift co-extensions for a class of row contractive Hilbert modules over the ball algebra. Finally, we obtain a curvature inequality for quotient modules, generalizing an earlier result for contractive Hilbert modules over the disk algebra.

The authors thank Scott McCollough for his comments and suggestions on an earlier draft of this paper.

\section{Preliminaries}

Let $\Omega \subseteq \mathbb C^m$ be a bounded, connected open set.  Fix an inner product on the algebra $\alg{\Omega}$, the completion in the supremum norm on $\Omega$ of the functions holomorphic on a neighborhood of the closure of $\Omega$. The completion of $\alg{\Omega}$ with respect to this inner product is a Hilbert space which we call
$\cl{M}$.  It is natural to assume that the module action
$\alg{\Omega} \times \alg{\Omega} \to \alg{\Omega}$ extends
continuously to $\alg{\Omega} \times \cl{M} \to \cl{M}$. Thus
a {\em Hilbert module} $\mathcal{M}$ over $\mathcal{A}(\Omega)$ is a Hilbert
space with a multiplication $\mathcal{A}(\Omega) \times \mathcal{M} \to
\mathcal{M}$ making $\mathcal{M}$ into a unital module over
$\mathcal{A}(\Omega)$ and such that multiplication is continuous.
Every cyclic or singly-generated bounded Hilbert module over $\mathcal A(\Omega)$ is obtained as a Hilbert space completion of $\mathcal A(\Omega)$.
Using the closed graph theorem one can show the existence of a
constant $\alpha$ such that
$$
\| f h\|_\mathcal{M} \leq \alpha \|f\|_{\mathcal{A}(\Omega)}
\|h\|_\mathcal{M},\, f\in \mathcal A(\Omega),\, h \in \mathcal M.
$$
One says that $\mathcal{M}$ is a {\it contractive} Hilbert module if $\alpha =
1$.

We assume that the module $\mathcal M$ is {\em quasi-free} (cf. \cite{eqhm}) of
multiplicity $n$ for $n \in \mathbb N$.
In particular, we assume that $\mathcal M$ is the completion of the algebraic tensor
product  $\mathcal A(\Omega) \otimes  \ell_n^2$ relative to an inner product so that
\begin{enumerate}
\item  multiplication by functions in $\mathcal A(\Omega)$ define bounded operators on $\mathcal M$,
\item the evaluation operators $e_w : \mathcal M \to \ell_n^2$ are locally uniformly
bounded on $\Omega$, and
\item  a sequence $\{f_k\} \subseteq \mathcal A(\Omega) \otimes \ell_n^2$ which
is Cauchy in the norm converges to $0$ if and only if $e_w(f_k)$ converges to $0$ for $w
\in \Omega$.
\end{enumerate}
As pointed out in \cite{eqhm}, using the identification of $\mathcal{M}$ with the completion of $\mathcal{A}(\Omega)$, $\clm$ can be realized as a space of
holomorphic functions on $\Omega$ which forms a kernel Hilbert space. In other words, ${\mathcal M}$ admits a
reproducing kernel $K:\Omega\times \Omega \to \C$ which is holomorphic
in the first variable and anti-holomorphic in the second one. It also has
the reproducing property:
$$
\langle h, K(\cdot, \w) \rangle = h(\w),~~h\in {\mathcal M},~\w\in \Omega.
$$

In some instances, such as the Drury-Arveson space, this definition does not apply. In such cases we define $\clm$ as the completion of the polynomial algebra $\mathbb{C}[\z]$ relative to an inner product on it assuming that each $p(\z)$ in $\mathbb{C}[\z]$ defines a bounded operator on $\clm$ but there is no uniform bound. Hence, in this case $\clm$ is a Hilbert module over $\mathbb{C}[\z]$.

Classical examples of contractive quasi-free Hilbert modules are:
\begin{enumerate}
\item[\rm (i)] the Hardy module $H^2(\D^m)$ (over the polydisk algebra
$\mathcal{A}(\D^m)$) which is the closure of the polynomials,
$\C[\z]$, in $L^2(\partial D^m)$ and
\item[\rm (ii)] the Bergman module $L_a^2(\Omega)$ (over the algebra
$\mathcal{A}(\Omega)$) which is the closure of $\mathcal{A}(\Omega)$
in $L^2(\Omega)$ with volume measure on $\Omega$.
\end{enumerate}

Let ${\mathcal L}(l^2_n)$ be the $C^*$-algebra of all bounded linear transformations on the Hilbert space $l^2_n$ of dimension $n$ for some $n \in \N$. We want to recall the notion of an $\cll(l^2_n)$-valued kernel function. Let $\Omega\subset \bb{C}^m$ be a bounded, connected open set. A function
$K:\Omega\times\Omega \to {\mathcal{L}}(l^2_n)$, holomorphic in the first variable and anti-holomorphic in the second one, satisfying
\begin{equation} \label{existence reprod}
\sum_{i,j=1}^p \inner{K(\w^{(i)},\w^{(j)})\zeta_j}{\zeta_i}_{l^2_n}
~\geq~0 ,~~\mbox{for~}\w^{(1)},\ldots,\w^{(p)}\in \Omega, ~~\zeta_1,\ldots,\zeta_p
\in l^2_n\mbox{~and}~ p \in  \mathbb{N},
\end{equation}
is said to be a {\em non negative definite (n.n.d.) kernel} on $\Omega$.
Given such an n.n.d. kernel $K$ on $\Omega$, it is easy to construct a \hs
$\mathcal{H}$ of functions on $\Omega$ taking values in $l^2_n$ with the
property that
\begin{equation} \label{reproducing property}
\inner{f(\w)}{\zeta}_{l^2_n} = \inner{f}{\kf{\w}\zeta}, \qfor \w \in
\Omega,~\zeta\in l^2_n, \qand f\in \mathcal{H}.
\end{equation}
The \hs $\mathcal{H}$ is simply the completion of the linear span $\clh^0$ of all functions of the form $\kf{\w}\zeta$, $\w \in \Omega$, $\zeta\in
l^2_n$. The inner product of two of the functions in $\clh^0$ is
defined by first setting
\begin{equation}\label{nndinner}
\inner{\kf{\w}\zeta}{\kf{\w^\prime}\eta} =
\inner{K(\w^\prime,\w)\zeta}{\eta}, \qfor \zeta,\eta\in l^2_n,
\qand \w,\w^\prime \in \Omega,
\end{equation}
and then extending to the linear span $\cl{H}^0$. This ensures the reproducing property (\ref{reproducing property}) of
$K$ on $\cl{H}^0$.
\begin{Remark} \label{nndkernel}
We point out that although the kernel $K$ is required merely to be {\em n.n.d.}, equation (\ref{nndinner}) defines a positive definite sesqui-linear form.
To see this, simply note that
$|\inner{f(\w)}{\zeta}| = |\inner{f}{\kf{\w}\zeta}|$
which is at most $\|f\|\inner{K(\w,\w)\zeta}{\zeta}^{1/2}$ by the
Cauchy - Schwarz inequality.
It follows that if $\|f\| = 0$ then $f(\w)=0$ for $\w \in \Omega$.
\end{Remark}
Conversely, let $\mathcal{H}$ be any Hilbert space of holomorphic functions on
$\Omega$ taking values in $l^2_n$.  Let $e_{\w} : \mathcal{H} \to
l^2_n$ be the evaluation functional defined by $e_{\w}(f) =
f(\w)$, $\w \in \Omega$, $f \in \mathcal{H}$.  If $e_{\w}$
is bounded for $\w \in \Omega$, then it admits a bounded
adjoint $e_{\w}^*: l^2_n \to \cl{H}$ such that $\inner{e_{\w}
f}{\zeta} = \inner{f}{e_{\w}^* \zeta}$ for all $f\in \cl{H}$ and
$\zeta\in l^2_n$.  A function $f$ in $\cl{H}$ is then orthogonal to
$e_{\w}^*(\cl{H})$ for all $\w \in \Omega$ if and only if $f=0$.  Thus the functions $f = \sum_{i=1}^p
e_{{\w}^{(i)}}^*(\zeta_i)$ with ${\w}^{(1)},
\ldots,{\w}^{(p)}\in \Omega,~\zeta_1,\ldots,\zeta_p \in l^2_n,\qand
p \in \mathbb{N},$ form a dense linear subset in $\cl{H}$.   Therefore, we have
$$
\|f\|^2 = \sum_{i,j=1}^p \inner{e_{{\w}^{(i)}}e_{{\w}^{(j)}}^*\zeta_j}
{\zeta_i},
$$
where $f=\sum_{i=1}^n e_{{\w}^{(i)}}^*(\zeta_i), ~{\w}^{(i)} \in \Omega,
~\zeta_i\in l^2_n$.
Since $\|f\|^2 > 0$, it follows that the kernel $K(\z,\w) = e_{\z} e_{\w}^*$
is non-negative definite as in (\ref{existence reprod}).  It is clear that
$K(\cdot,\w)\zeta \in \mathcal{H}$ for each $\w \in \Omega$ and $\zeta\in l^2_n$,
and that it has the reproducing property (\ref{reproducing property}).
\begin{Remark} \label{nonsingker}
If we assume that the evaluation functional $e_{\w}$ is surjective, then the
adjoint $e_{\w}^*$ is injective and it follows that
$\inner{K(\w,\w)\zeta}{\zeta} = \|e^*_{\w} \zeta \|^2 > 0$ for all non-zero vectors $\zeta\in l^2_n$.
\end{Remark}

For $1\leq i\leq m$, suppose that the \ops $M_i:{\mathcal H} \to
{\mathcal H}$ defined by $M_if(\w) = w_if(\w)$ for $f\in \mathcal H$ and $\w \in \Omega$, are bounded.  Then it is easy to verify that for each fixed $\w \in \Omega$, and $1\leq i \leq m$,
\begin{equation} \label{eigenspace}
M_i^* K(\cdot, \w)\eta = \bar{w}_i K(\cdot, \w)\eta \qfor \eta \in l^2_n.
\end{equation}

\begin{Remark}\label{eigenspace2}
As a consequence of (\ref{eigenspace}) we see that the vectors $\{ K(\cdot, \w^{(i)}) \eta_i\}_{i=1}^p$ for $\w^{(1)}, \ldots, \w^{(p)} \in \Omega$ and $\eta_i \in l^2_n,\, p>0,$ are linearly independent if the $\w^{(i)}$ are distinct or if each subset of  $\eta_i$ corresponding to equal $\w^{(i)}$ are linearly independent.
\end{Remark}

One may impose conditions on a kernel function $K:\Omega \times
\Omega \to {\mathcal L}(l^2_n)$ to ensure the boundedness of each of
the multiplication \ops $M_1, \ldots, M_m$ on the associated \rk \hsp.
Let $\{\varepsilon_1, \ldots , \varepsilon_n\}$ be an orthonormal  basis for $l^2_n$.  Let $\cl{H}^0$ be the linear span of the vectors $\{K(\cdot,\w)\varepsilon_1, \ldots , K(\cdot, \w)\varepsilon_n : \w = (w_1, \ldots , w_m) \in \Omega\}$ assuming $K$ satisfies the condition in Remark \ref{eigenspace2}. Clearly, the linear subspace
${\mathcal{H}}^{\circ} \subseteq {\mathcal{H}}$ is dense in the Hilbert space ${\mathcal{H}}$.   Define a map $T_\ell, 1 \leq \ell
\leq m,$ by the formula $T_\ell K(\cdot,\w) \varepsilon_j =
\bar{w}_\ell K(\cdot, \w) \varepsilon_j$ for $1\leq j \leq n$, and $\w \in \Omega$ which is well defined by the assumption above in Remark \ref{eigenspace2}.  The following well known lemma gives a criterion for the boundedness of the adjoint of the coordinate operators $T_\ell, 1\leq \ell \leq m$. We include a proof for completeness.

\begin{Lemma} \label{bdd}
The densely defined map $T_\ell:{\mathcal H}^0 \to {\mathcal H}^0
\subseteq \mathcal H$, $1 \leq \ell \leq m$, is bounded if
and only if for some positive constants $c_l$ and for all $k\in \N$
$$
\sum_{i,j=1}^k \inner{(c_l^2 -\w_\ell^{(j)}\bar{\w}_\ell^{(i)})
K(\w^{(j)},\w^{(i)})x_i}{x_j} \geq 0.
$$
for $x_1, \ldots, x_k \in l^2_n$ and $\w^{(1)}, \ldots,
\w^{(k)} \in \Omega$.  If the map $T_\ell$ is bounded, then it is
the adjoint of the multiplication operator $M_\ell :\mathcal H
\to \mathcal H$, $1 \leq \ell \leq m$ and $\|T_l\|$ is the smallest $c_l$ for which the positivity condition holds.
\end{Lemma}
\begin{proof} The proof in the forward direction amounts to a verification
of the positivity condition in the statement of the Lemma.  To verify this, fix $\ell$,
$1 \leq \ell \leq m$, and note
that if $T_\ell$ is bounded on $\cl{H}^0$, then we must have
\begin{eqnarray*}
\|T_\ell (\sum_{i=1}^k K(\cdot, \w^{(i)}) x_i)\|^2 &=&
\|\bar{w}_\ell^{(i)}\sum_{i=1}^k K(\cdot,\w^{(i)}) x_i \|^2 \\
&=& \inner{\bar{w}^{(i)}_\ell \sum_{i=1}^k K(\cdot, \w^{(i)}) x_i}
{\bar{w}_\ell^{(j)}\sum_{j=1}^k K(\cdot, \w^{(j)}) x_j}\\
&=&\sum_{i,j=1}^k \bar{w}_\ell^{(i)} w_\ell^{(j)}
\inner{K(\w^{(j)},\w^{(i)})x_i}{x_j}\\
&\leq& \|T_l\|^2 \inner{\sum_{i=1}^k K(\cdot, \w^{(i)}) x_i}{\sum_{j=1}^k K(\cdot,
\w^{(j)}) x_j}\\
&=& \|T_l\|^2 \sum_{i,j=1}^k \inner{K(\w^{(j)},\w^{(i)})x_i}{x_j},
\end{eqnarray*}
for all possible vectors $x_1, \ldots,
x_k \in l^2_n$, $\w^{(1)}, \ldots, \w^{(k)} \in \Omega$.  However,
combining the last two lines, we obtain the
positivity condition of the Lemma.  On the other hand, if the
positivity condition is satisfied for some positive constant $c_l$, then the preceding calculation shows
that $T_\ell$,  $1 \leq \ell \leq m$, is bounded on $\cl{H}^0 \subseteq
\cl{H}$.  Therefore, it defines a bounded linear operator on all of
$\cl{H}$ with $\|T_\ell\| \leq c_\ell$,  $1 \leq \ell \leq m$.

Recall that if the operator $M_\ell$, defined to be multiplication by the co-ordinate function
$z_\ell$, is bounded, then $K(\cdot, \w)x$ is an eigenvector with
eigenvalue $\bar{w_l}$ for the adjoint $M_\ell^*$ on $\cl{H}$.
This proves the last statement of the Lemma with the relation between $c_l$ and $\|T_l\|$ being straightforward.
\end{proof}
We abbreviate the positive definiteness condition of Lemma \ref{bdd},  namely,
$$\sum_{i,j=1}^k \inner{(c_l^2 - w_\ell^{(j)}\bar{w}_\ell^{(i)})
K(\w^{(j)},\w^{(i)})x_i}{x_j} \geq 0,\mbox{ for } x_1, \ldots, x_k
\in l^2_n\mbox{ and }\w^{(1)}, \ldots, \w^{(k)} \in \Omega$$
to saying that
$(c_l^2 -z_\ell\bar{\omega}_\ell) K(\z,\w)$
is positive definite for each $\ell$,  $1 \leq \ell \leq m$.

\begin{Remark}
A module action by $\mathbb{C}[\z]$ on the reproducing kernel Hilbert space $\clh$ with kernel function $k(\z, \w)$ is said to be {\it compatible} if $M^*_{z_i} k(\cdot, \w) = \bar{w_i} k(\cdot, \w)$ for $\w \in \Omega$ and $1 \leq i \leq m$. (Note that compatibility for $K$ implies the conclusion of Remark \ref{eigenspace2}). A kernel Hilbert space need not posses a compatible module structure but if it does, it is unique. Consider the following example. If $f : \mathbb{D} \raro \mathbb{C} \backslash\ \{0\}$ is a holomorphic function, then $k(z, w) = f(z) \overline{f(w)}$ is non-negative definite. However, the Hilbert space $\clh$ of functions $\{k(\cdot, w) : w \in  \mathbb{D}\}$ consists of scalar multiples of $f$ and hence $\clh$ is one dimensional. If $M_z$ is defined on $\clh$, then $M_z^* k(\cdot, w) = \bar{w} k(\cdot, w)$ for $w \in \mathbb{D}$ and hence $k(\cdot, w)$ is an eigenvector for the eigenvalue $\bar{w}$. Thus for distinct $w, w' \in \mathbb{D}$, the vectors $k(\cdot, w)$ and $k(\cdot, w')$ are linearly independent which contradicts the fact that $\clh$ is one dimensional. Thus no compatible module action can be defined on $\clh$.
\end{Remark}

\section{Co-extensions and kernel functions}
Let $\clr \subseteq \mbox{Hol}(\Omega, \mathbb{C})$ be a reproducing kernel Hilbert space with the scalar kernel function $k : \Omega \times \Omega \raro \mathbb{C}$. (Note that the containment of $\clr$ in $\mbox{Hol}(\Omega, \mathbb{C})$ determines the kernel function $k(\z, \w)$ and vice versa.) Let $\cle$ be a separable Hilbert space so that the Hilbert space tensor product $\clr \otimes \cle \subseteq \mbox{Hol}(\Omega, \cle)$ is a reproducing kernel Hilbert space with the kernel function $(k \otimes I_{\cle})(\bm{z}, \bm{w}) = k(\bm{z}, \bm{w}) I_{\cle} \in \cll(\cle)$.

Let $\clm$ be a quasi-free Hilbert module of multiplicity $n \,(1 \leq n < \infty)$ over $A(\Omega)$ for which the evaluation operator $e_{\w}$ is surjective for $\w \in \Omega$. We recall, as shown in the previous section, that the kernel function $K_{\clm}$ of $\clm$ is given by $$K_{\clm}(\bm{z}, \bm{w}) = e_{\bm{z}} e_{\bm{w}}^* : \Omega \times \Omega \raro \cll(l^2_n).$$

Now, let $\clm$ be a Hilbert module isomorphic to $(\clr \otimes \cle)/ \cls$ for some submodule $\cls$ of $\clr \otimes \cle$, or equivalently, $\clm$ has an isometric co-extension to $\clr \otimes \cle$. Consequently, we have the short exact sequence of Hilbert modules
$$0 \raro  \cls \stackrel{i} \raro \clr \otimes \cle \stackrel{\pi} \raro \clm \raro 0,$$ where the second map is the inclusion $i$ and the third map is the quotient map $\pi$ which is a co-isometry. For each $\bm{w}$ in $\Omega$, define the ideal $I_{\bm{w}} = \{ \varphi \in A(\Omega) : \varphi(\bm{w}) = 0 \}$ (or $\{ p(\bm{z}) \in \mathbb{C} [\bm{z}] : p(\bm{w}) = 0\}$). Also recall that $\clm/ \clm_{\bm{z}} \cong \clm \otimes_{A(\Omega)} \mathbb{C}_{\bm{z}} \cong \mathbb{C}_{\bm{z}} \otimes l^2_n \cong l^2_n$, where $\clm_{\bm{z}}$ is the closure of $I_{\bm{z}} \clm$ in $\clm$.

\begin{Theorem}\label{TH1}
Let $\clr \subseteq \mbox{Hol}(\Omega, \mathbb{C})$ be a reproducing kernel Hilbert module over $A(\Omega)$ (or over $\mathbb{C}[\bm{z}]$) with the scalar kernel function $k$ and $\clm$ be a quasi-free Hilbert module over $A(\Omega)$ (or over $\mathbb{C}[\bm{z}]$) of multiplicity $n$. Then $\clm$ has an isometric co-extension to $\clr \otimes \cle$ for the Hilbert space $\cle$, if and only if there is a holomorphic map $\pi_{\bm{z}} \in \clo(\Omega, \cll(\cle, l^2_n))$ such that $$K_{\clm}(\bm{z}, \bm{w}) = k(\bm{z}, \bm{w}) \pi_{\bm{z}} \pi_{\bm{w}}^*, \quad \quad \bm{z}, \bm{w} \in \Omega.$$
\end{Theorem}

\NI \textsf{Proof.} To prove the necessity part, we localize the exact sequence of Hilbert modules
$$0 \raro \cls \raro \clr \otimes \cle \raro \clm \raro 0,$$ at $\bm{z}$, and obtain the following diagram
\setlength{\unitlength}{3mm}
 \begin{center}
 \begin{picture}(40,16)(0,0)
\put(2,3){$ \cls/I_{\bm{z}} \cls$} \put(10,3){$ (\clr \otimes \cle)/I_{\bm{z}} (\clr \otimes \cle) $} \put(26,3){$ \clm/I_{\bm{z}} \clm$} \put(36,3){0}

 \put(6.6,2.1){$i_{\bm{z}}$} \put(23.6, 2.1){$\pi_{\bm{z}}$}

 \put(3,6.5){$N_{\bm{z}}$} \put(15,6.4){$P_{\bm{z}}$} \put(28,6.4){$Q_{\bm{z}}$}

 \put(-5, 9.5){0} \put(2,9.5){$\cls$}\put(12,9.5){$\clr \otimes \cle$} \put(27,9.5){$\clm$} \put(36,9.5){0}
 \put(5.6,10.5){$i$} \put(22.6, 10.5){$\pi$}

 \put(-4,10){ \vector(1,0){5}} \put(5.5,3.5){ \vector(1,0){3}} \put(4.5,10){ \vector(1,0){5}} \put(21.5,10){ \vector(1,0){4}} \put(22,3.5){ \vector(1,0){3}} \put(31,3.5){ \vector(1,0){4}}  \put(31,10){ \vector(1,0){4}}
 \put(2.4,9.2){ \vector(0,-1){5}} \put(14,9.2){ \vector(0,-1){5}} \put(27,9.2){ \vector(0,-1){5}}

 \end{picture}
 \end{center}

\NI which is commutative with exact rows for all $\bm{w}$ in $\Omega$ (see \cite{DP}). Here $N, P, Q$ are the natural co-isometric or quotient module maps. If we identify $\clm/ I_{\bm{z}} \clm$ with $l^2_n$ and $(\clr \otimes \cle)/I_{\bm{z}} (\clr \otimes \cle)$ with $\cle$, then the kernel functions of $\clm$ and $\clr \otimes \cle$ are given by $Q_{\bm{z}} Q_{\bm{w}}^*$ and $P_{\bm{z}} P_{\bm{w}}^*$, respectively. Moreover, since $Q_{\bm{w}} \pi = \pi_{\bm{w}} P_{\bm{w}}$ for all $\bm{w} \in \Omega$, we have that $Q_{\bm{z}} \pi \pi^* Q_{\bm{w}} = \pi_{\bm{z}} P_{\bm{z}} P_{\bm{w}}^* \pi_{\bm{w}}^*$. Using the fact that $\pi \pi^* = I_{\clm}$ and $P_{\bm{z}} P_{\bm{w}}^* = k(\bm{z}, \bm{w}) \otimes I_{\cle}$, we infer that $$Q_{\bm{z}} Q_{\bm{w}}^* = k(\bm{z}, \bm{w}) \pi_{\bm{z}} \pi_{\bm{w}}^*, \quad \quad \bm{z}, \bm{w} \in \Omega.$$

Conversely, assume that for a given quasi-free Hilbert module $\clm$, the kernel function of $\clm$ has the factorization $$K_{\clm}(\bm{z}, \bm{w}) = k(\bm{z}, \bm{w}) \pi_{\bm{z}} \pi_{\bm{w}}^*, \quad \quad \bm{z}, \bm{w} \in \Omega,$$ for some function $\pi : \Omega \raro \cll(\cle, l^2_n)$. Note that if the function $\pi$ satisfies the above equality then it is holomorphic on $\Omega$. Now, we define a linear map $X : \clm \raro \clr \otimes \cle$ so that $$X Q_{\bm{z}}^* \eta = P_{\bm{z}}^* \pi_{\bm{z}}^* \eta, \quad \quad \eta \in l^2_n.$$  It then follows that

\begin{equation*}
\begin{split}
\langle X(Q^*_{\bm{w}} \eta), X (Q^*_{\bm{z}} \zeta) \rangle = & \langle P_{\bm{w}}^* \pi_{\bm{w}}^* \eta, P^*_{\bm{z}} \pi^*_{\bm{z}} \zeta \rangle = \langle \pi_{\bm{z}} P_{\bm{z}} P^*_{\bm{w}} \pi^*_{\bm{w}} \eta, \zeta \rangle  = \langle Q_{\bm{z}} Q^*_{\bm{w}} \eta, \zeta \rangle \\& = \langle Q^*_{\bm{w}} \eta, Q^*_{\bm{z}} \zeta \rangle,
\end{split}
\end{equation*}
for all $\eta, \zeta \in l^2_n$. Therefore, since $\{Q_{\bm{z}}^* \eta : \bm{z} \in \Omega, \eta \in l^2_n\}$ is a total set of $\clm$, then $X$ extends to a bounded isometric operator. Moreover, by the reproducing property of the kernel function, it follows that  $$M_{\varphi}^* X (Q_{\bm{z}}^* \eta) = M_{\varphi}^* P^*_{\bm{z}} (\pi^*_{\bm{z}} \eta) = \overline{\varphi(\bm{z})} P^*_{\bm{z}} \pi^*_{\bm{z}} \eta = \overline{\varphi(\bm{z})} X(Q^*_{\bm{z}} \eta) = XQ^*_{\bm{z}} (\overline{\varphi(\bm{z})} \eta) = X M_{\varphi}^* (Q_{\bm{z}}^* \eta),$$ for all $\varphi \in A(\Omega)$ and $\eta \in l^2_n$. Hence, $X^* \in \cll(\clr \otimes \cle, \clm)$ is a module map. \qed

As an application of the above theorem, we have the main result of this section

\begin{Theorem}\label{TH2}
Let $\clm$ be a quasi-free Hilbert module over $A(\Omega)$ (or over $\mathbb{C}[\bm{z}]$) of multiplicity $n \in \mathbb{N}$ and $\clr$ be a reproducing kernel Hilbert module over the same algebra. Let $k$ be the kernel function of $\clr$ and $K_{\clm}$ be that of $\clm$. Then $\clm$ has an isometric co-extension to $\clr \otimes \cle$ for some Hilbert space $\cle$ if and only if $$K_{\clm}(\bm{z}, \bm{w}) = k(\bm{z}, \bm{w}) \tilde{K}(\bm{z}, \bm{w}),$$ for some positive definite kernel $\tilde{K}$ over $\Omega$. Moreover, if $k^{-1}$ is defined, then the above conclusion is true if and only if $k^{-1} K_{\clm}$ is a positive definite kernel.
\end{Theorem}
\NI \textsf{Proof.} Observe that, the equality in the statement tells us that the kernel function $\tilde{K}$ is $\cll(l^2_n)$-valued, where $n$ is the multiplicity of $\clm$. Since the necessary part follows from the previous theorem, all that remains to be shown is that the factorization given in the statement yields an isometric co-extension. If $\tilde{K}$ is given to be a positive definite kernel, then we let $\clh(\tilde{K})$ be the corresponding reproducing kernel Hilbert space (which is not necessarily a module over $A(\Omega)$ or  even over $\mathbb{C}[\bm{z}]$). Then we set $\cle = \clh(\tilde{K})$ and let $\pi_{\bm{z}} = \mbox{e}_{\bm{z}} \in \cll(\cle, l^2_n)$ be the evaluation operator for the reproducing kernel Hilbert space $\clh(\tilde{K})$. Then the fact that $\clr \otimes \cle$ is an isometric co-extension of $\clm$ follows immediately from the previous theorem since $\tilde{K}(\z, \w) = \pi_{\z} \pi^*_{\w}$. \qed

\begin{Remark}
If the kernel function $\tilde{K}$ defines a Hilbert space of holomorphic functions on $\Omega$ invariant under $\mathbb{C}[\z]$, then one can identify $\clm$ canonically with the Hilbert module tensor product, $\clr \otimes_{\mathbb{C}[\mathbf z]} \clh(\tilde{K})$, which yields an explicit representation of the co-isometry from the co-extension space $\clr \otimes \clh(\tilde{K})$ to $\clm$.
\end{Remark}

If $\clh(\tilde{K})$ is not invariant under the action of $\mathbb{C}[\z]$, we can still describe the co-extension space explicitly using a construction of Aronszajn. Let $\clm_1$ and $\clm_2$ be Hilbert spaces of holomorphic functions on $\Omega$ so that they possess reproducing kernels $K_1$ and $K_2$, respectively. Assume that the natural action of $\mathbb{C}[\z]$ on the Hilbert space $\clm_1$ is continuous; that is, the map $(p, h) \raro p h$ defines a bounded operator on $\clm_1$ for $p \in \mathbb{C}[\z]$. (We make no such assumption about the Hilbert space $\clm_2$.) Now, $\mathbb{C}[\z]$ acts naturally on the Hilbert space tensor product $\clm_1 \otimes \clm_2$ via the map $$(p, (h \otimes k)) \mapsto p \cdot h \otimes k, \,\; p \in \mathbb{C}[\z],\, h \in \clm_1,\, k \in \clm_2.$$
The map $h  \otimes k \mapsto h k$ identifies the Hilbert space $\clm_1 \otimes \clm_2$ as a reproducing kernel Hilbert space of holomorphic functions on $\Omega \times \Omega$ \cite{aron}. The module action is then the pointwise multiplication $(p, hk) \raro (ph) k$, where $((p h) k)(\z_1, \z_2) =  p(\z_1) h(\z_1) k(\z_2)$, $\z_1, \z_2 \in \Omega$.

We denote by $\clh$ the Hilbert module obtained by the Hilbert space $\clm_1 \otimes \clm_2$ with the above module action over $\mathbb{C}[\z]$. Let $\bigtriangleup \subseteq \Omega \times \Omega$ be the diagonal subset $\{(\z,\z): \z \in \Omega \}$ of $\Omega \times \Omega$. Let $\cls$ be the maximal submodule $\mathcal S$ of $\clm_1 \otimes \clm_2$ functions in $\clm_1 \otimes \clm_2$ which vanish on $\bigtriangleup$. Thus
$$ 0 \longrightarrow  \cls \stackrel{X} \longrightarrow \clm_1 \otimes \clm_2 \stackrel{Y} \longrightarrow  \cl{Q} \longrightarrow 0,$$
is a short exact sequence, where $\cl{Q} = (\clm_1 \otimes \clm_2) / \mathcal S$, $X$ is the inclusion map and $Y$ is the natural quotient map. (Note that $\cls = \{0\}$ is possible in which case $\mathcal Q = \clm_1 \otimes \clm_2$.) One can appeal to an extension of an earlier result of Aronszajn \cite{aron} to
analyze the quotient module $\mathcal Q $ when the given module is a reproducing kernel Hilbert space. The reproducing kernel of $\mathcal H$ is then the point-wise product $K_1 (\z,\w) K_2 (\u,\v)$ for $\z,\w,\u,\,\v \in \Omega$.
Set $\cl H_{\text{res}} = \{f_{|\,\bigtriangleup} : f\in \cl H\}$ and $\|f_{|\,\bigtriangleup}\|_{\text{res}} = \inf\{\|g\|_{\cl H} : g\in \cl H, g_{|\,\bigtriangleup} \equiv f_{|\, \bigtriangleup}\}$.

\begin{Proposition}[Aronszajn] The module $\cl H_{\text{res}}$ is a kernel
Hilbert module consisting of holomorphic functions on $\bigtriangleup$.  Its kernel function, $K_{\text{res}\, \bigtriangleup}$, is the restriction to $\bigtriangleup$ in both sets of variables of the original kernel function $K$ for the Hilbert module $\mathcal H$. Moreover, the quotient module $\cl Q$ corresponding to the submodule in $\mathcal H$ of functions vanishing on $\bigtriangleup$ is isometrically isomorphic to $\cl H_{\text{res}}$.
\end{Proposition}

We now reformulate this result to apply to the context of Theorem \ref{TH2}.

\begin{Proposition} \label{main}
Let $\mathcal M$ be a Hilbert module over the polynomial algebra $\mathbb{C}[\z]$ and $K_{\clm}$ be its reproducing kernel defined on the domain $\Omega \subseteq \mathbb{C}^m$. Suppose $K_{\clm}$ is the point-wise product of two positive definite kernels $K_1$ and $K_2$ on $\Omega \times \Omega$ and $\clm_1$ and $\clm_2$ are the corresponding kernel Hilbert spaces of holomorphic functions on $\Omega$. Assume that the polynomial algebra $\mathbb{C}[\z]$ acts on $\mathcal M_1$ continuously. Then the compression of the natural action of $\mathbb{C}[\z]$ on $\mathcal M_1 \otimes \mathcal M_2$ given by the operators $M_p\otimes I_{\clm_2}$, $p \in \mathbb{C}[\z]$, to $\mathcal Q \subseteq \mathcal M_1 \otimes \mathcal M_2$, coincides with the action of $\mathcal A(\Omega)$ on $\mathcal M$; that is, $\mathcal M_1 \otimes \mathcal M_2$ is an isometric co-extension of $\mathcal M$.
\end{Proposition}

Thus the dilation space in Theorem \ref{TH2} can be realized as a ``reduced'' Hilbert module tensor product which coincides with the module tensor product when the coefficient space is also a module over $\mathbb{C}[\z]$.

In \cite{rgdgm} this result was used to analyze the quotient $H^2(\bb D^2)/[z_1-z_2]$.
Since the kernel function for $H^2(\bb D^2)$ is
$\frac1{(1-z_1\bar w_1)}$ $\frac1{(1-z_2\bar w_2)}$,
restricting the kernel function to $\bigtriangleup = \{(w, z) : w-z=0\}$ and using the $(u,v)$ coordinates; that is, $(u = \frac{z_1+z_2}{2}, w = \frac{z_1-z_2}{2}$), we obtain that $
K_\mathcal Q(u,u') = \frac{1}{(1-u\bar{u}')^2}$ for $u,u'$ in $\{(w, z) : w-z=0\}$. Since this is the
kernel function for the Bergman space $L_a^2(\mathbb D)$, the
quotient module in this case is isometrically isomorphic to the Bergman module. Thus we obtain an isometric co-extension of the Bergman shift. (Note that this extension agrees with the one obtained in the Sz.-Nagy - Foias model.) Thus the extension of Aronszajn's result enables one to obtain the kernel function for the quotient module and from it, one can construct the Hilbert space itself.

We end this section with the following remarks.

\begin{Remark} Applying the same argument, as used for the necessity part of Theorem \ref{TH1}, to the left hand square of the diagram yields the following relation between the kernel functions for $\cls$ and $\clr \otimes \cle$: $$N_{\bm{z}} N_{\bm{w}}^* = i_{\bm{z}} P_{\bm{z}} P_{\cls} P_{\bm{w}}^* i_{\bm{w}},$$
\NI where $P_{\cls}$ is the projection onto $\cls \subseteq \clr \otimes \cle$.  Thus the kernel function for $\mathcal S$ is related to that of the Hardy module but because $P_\mathcal S$ does not "commute" with these terms, the relationship is more complicated.
\end{Remark}

\begin{Remark} Using the relation $N_{\bm{z}} N_{\bm{w}}^* = i_{\bm{z}} P_{\bm{z}} P_{\cls} P_{\bm{w}}^* i_{\bm{z}},$ we can compare formulas for the kernel functions, where $\cls$ consists of functions vanishing on a hypersurface (cf. \cite{DMV}). We hope to return to such applications at a later time.
\end{Remark}

\section{Co-extensions and Hereditary functional calculus}

Let $p$ be a polynomial in the $2 m$ variables $\z, \w$, where the $\z$ -variables all commute and the $\w$-variables all commute but we assume nothing about the relation of the $\z$ and $\w$ variables. For any commuting $m$ - tuple of operators $\mathbf T = (T_1, \ldots , T_m)$, following Agler (see \cite{Agl}) we define the value of $p$ at $\mathbf T$ using the hereditary functional calculus:
$$
p(\mathbf T, \mathbf T^*) = \sum_{I,J} a_{I,J} \mathbf T^I {\mathbf T^*}^J,
$$
where $p(\z,\w) = \sum_{I,J} a_{I,J} \z^I \w^J$ and $I=(i_1, \ldots,i_m)$, $J = (j_1, \ldots ,j_m)$ are multi-indices of length $m$. Here, in the ``non-commutative polynomial'' $p(\z, \w)$, the ``$\z$'s'' are all placed on the left, while the ``$\w$'s'' are placed on the right.

Let  $\mathbf{M}= (M_1,\ldots, M_m)$ be the $m$ - tuple of multiplication operators on a reproducing kernel Hilbert space $\mathcal{H}$ defined on the polydisk $\D^m$. Let $K$ be the reproducing kernel for $\mathcal{H}$.  Let $\mathbb S_{\mathbb D^m}^{-1}(\z,\w) = \prod_{\ell = 1}^m \big (1-z_\ell \bar{w}_\ell \big ) = \sum_{0 \leq |I| =  |J| \leq m} \z^I \bar{\w}^J$,  for $\z,\w \in \mathbb D^m$; that is, $\mathbb S_{\mathbb D^m}$ is the S{z}eg\"{o} kernel for the polydisk $\mathbb D^m$. Observe that $K(\z, \w)$ is in $\cll(l^2_n)$ and hence a calculation shows
\begin{eqnarray*}
\big ( \mathbb S_{\mathbb D^m}^{-1}(\mathbf M, \mathbf M^*)\big ) K(\z , \w) &=&\Big ( \big ( \sum_{0 \leq |I| =  |J| \leq m} \z^I \bar{\w}^J \big ) (\mathbf M, \mathbf M^*) \Big ) K(\z , \w) \\
&=& \Big (\sum_{0\leq |I| = |J| \leq m} \mathbf M^I {\mathbf M^J}^* \Big ) K(\z,\w) \\
&=& \Big ( \sum_{0\leq |I|  = |J| \leq m} \z^I \bar{\w}^J \Big ) K(\z,\w) \\
&=& \mathbb S_{\mathbb D^m}^{-1}(\z,\w) K(\z,\w).
\end{eqnarray*}
Hence, $\mathbb S_{\mathbb D^m}^{-1}(\mathbf M, \mathbf M^*) \geq 0$ if and only if $\mathbb S_{\mathbb D^m}^{-1}(\z,\w) K(\z,\w)$ is a non-negative definite kernel, which implies the following.

\begin{Theorem}\label{TH3} The positivity of the operator
$\mathbb S_{\mathbb D^m}^{-1}(\mathbf M, \mathbf M^*)$, defined via the hereditary functional calculus, on the Hilbert space $\mathcal H$ possessing the reproducing kernel $K$, is equivalent to the factorization
$$
K(\z,\w) = \mathbb S_{\mathbb D^m}(\z,\w) Q(\z,\w),\,\, \z,\w \in \mathbb D^m,
$$
where $Q$ is some positive definite kernel on the polydisk $\mathbb D^m$.
\end{Theorem}

As an application of Theorem \ref{TH2} and Theorem \ref{TH3}, we obtain half of our main result for the Hardy module $H^2(\mathbb{D}^m) \otimes \mathcal Q$.

\begin{Theorem} \label{dil}
Let $\mathcal H$ be a Hilbert module over the polydisk algebra $\mathcal A(\mathbb D^m)$ with reproducing kernel $K(\z, \w)$.  Assume that the operator $\mathbb S_{\mathbb D^m}^{-1}(\mathbf M,\mathbf M^*)$, defined via the hereditary functional calculus, is positive.  Then $\mathcal H$ can be realized as a quotient module of the Hardy module $H^2(\mathbb D^m) \otimes \mathcal Q$ over
the algebra $\mathcal A(\mathbb D^m)$ for some Hilbert space $\mathcal Q$, and conversely. Hence $\clh$ has an isometric co-extension to $H^2(\mathbb{D}^m) \otimes \mathcal Q$ in this case.
\end{Theorem}

A closely related result was obtained by Athavale (Theorem 2.6 in \cite{Ath}) but with a different proof.

Notice that any $m$ - tuple  of doubly commuting contractions on a functional Hilbert space over $\mathcal A(\mathbb{D}^m)$ satisfies the hypothesis of Theorem \ref{dil}. Thus, we recover the result of Sz.-Nagy and Foias (cf. \cite{Na-Fo}) in this situation.  But the class covered by the theorem is much larger.

In particular, $\clm^m = \clm \otimes_{\cla(\mathbb{D}^m)} \cdots \otimes_{\cla(\mathbb{D}^m)} \clm$ always possesses a dilation to the Hardy module $H^2(\mathbb D^m)\otimes \mathcal E$, where $\mathcal E$ is some Hilbert space, if $\clm$ is contractive. Let $\mathbf M$  denote the $m$ - tuple of commuting contractions on the Hilbert module $\mathcal M$  possessing the reproducing kernel $K$.  The contractivity condition implies that $K(\z,\w) = (1-z_\ell \bar{w}_\ell)^{-1} Q_\ell(\z,\w)$ for some positive definite kernel $Q_\ell$ and for each $\ell = 1,2,\ldots , m$. Thus
$$
K^m(\z,\w) = \mathbb S_{\mathbb D^m}(\z,\w) Q(\z,\w),\,\,\z,\w \in \mathbb D^m,
$$
where $Q=\prod_{\ell=1}^m Q_\ell$.
Thus the Hilbert module $\clm^m = \clm \otimes_{\cla(\mathbb{D}^m)} \cdots \otimes_{\cla(\mathbb{D}^m)} \clm$ corresponding to the positive definite kernel  $K^m$ is contractive and admits the kernel
$\mathbb S_{\mathbb D^m}$ as a factor, as shown above. This shows that $\clm^m$ has an isometric co-extension to $H^2_{\mathcal Q}(\mathbb{D}^m)$, where $\mathcal Q$ is the reproducing kernel Hilbert space for the kernel $Q$.

We now state our main result for the Hardy module over the polydisk algebra and investigate the uniqueness of an isometric co-extension $X : \clr \otimes \cle \raro \clm$ which is \textit{minimal}. Recall this means that the smallest reducing submodule of $\clr \otimes \cle$ containing $X^* \clm$ is $\clr \otimes \cle$ itself. Although we believe such a uniqueness result holds for a more general class of quasi-free Hilbert modules $\clr$, we prove it only in the cases when $\clr = H^2_m$ and $\clr = H^2(\mathbb{D}^m)$ for $m \geq 1$. When $\clr = H^2_m$, the result follows from the uniqueness of the minimal isometric dilation by Arveson (see \cite{Ar}). We prove the case when $\clr = H^2(\mathbb{D}^m)$. This result was proved in \cite{DF} for the case of multiplicity one. The proof is based on operator theory exploiting the fact that the co-ordinate multipliers define doubly commuting isometries.

\begin{Theorem}\label{cor1}
If $\clh$ is a contractive reproducing kernel Hilbert space over $\cla (\mathbb{D}^m)$, then $\clh$ has an isometric $H^2 (\mathbb{D}^m) \otimes {\mathcal E}$ co-extension for some Hilbert space $\mathcal  E$ if and only if $\mathbb S_{\mathbb D^m}^{-1}(\mathbf M,\mathbf M^*) \geq 0$ or, equivalently, if and only if $\mathbb S_{\mathbb D^m}^{-1} K \geq 0$, where $K$ is the kernel function for $\clh$. Note that this means there exists a co-isometric module map $Y : H^2 (\mathbb{D}^m) \otimes {\mathcal  E} \raro \clh$. Moreover, if an $H^2 (\mathbb{D}^m) \otimes {\mathcal  E}$ isometric co-extension exists, then the minimal one is unique.
\end{Theorem}
\NI\textsf{Proof.} Using Theorem \ref{dil},  every thing is proved except the uniqueness part. Suppose $X_i : H^2(\mathbb{D}^m) \otimes \cle_i \raro \clh$ are co-isometric module maps for $i = 1, 2$ so that the co-extensions are minimal. We must exhibit a unitary module map $$V : H^2(\mathbb{D}^m) \otimes \cle_1 \raro H^2(\mathbb{D}^m) \otimes \cle_2$$ so that $X_1 = X_2 V$. Such a map $V$ must have the form $I_{H^2(\mathbb{D}^m)} \otimes V_0$ for some unitary operator $V_0 : \cle_1 \raro \cle_2$, which will conclude the proof of uniqueness.

\NI Let $\bm{\alpha} = (\alpha_1, \ldots, \alpha_m)$ be a multi-index with $\alpha_i \in \mathbb{Z}_+, i = 1, \ldots, m$, and $|\bm{\alpha}| = \alpha_1 + \cdots + \alpha_m$. For $N \in \mathbb{N}$, let $\{f_{\bm{\alpha}}\}_{|\bm{\alpha}| \leq N}$ be vectors in $\clh$. We want to show that the map $$V(\sum_{|\bm{\alpha}| \leq N} M_{\z^{\bm{\alpha}}} X^*_1 f_{\bm{\alpha}}) = \sum_{|\bm{\alpha}| \leq N} M_{\z^{\bm{\alpha}}} X^*_2 f_{\bm{\alpha}},$$ extends to a unitary module map from $H^2(\mathbb{D}^m) \otimes \cle_1$ to $H^2(\mathbb{D}^m) \otimes \cle_2$. The first step is to show that this $V$ preserves the inner products, for which it is enough to show $$\langle  M_{\z^{\bm{\alpha}}} X^*_1 f_{\bm{\alpha}},  M_{\z^{\bm{\beta}}} X^*_1 f_{\bm{\beta}} \rangle = \langle  M_{\z^{\bm{\alpha}}} X^*_2 f_{\bm{\alpha}},  M_{\z^{\bm{\beta}}} X^*_2 f_{\bm{\beta}} \rangle,$$ for all $\bm{\alpha}$ and $\bm{\beta}$. We define multi-indices $\bm{\gamma}$ and $\bm{\mu}$ so that
\[ {\gamma}_i  = \left\{ \begin{array}{cc}
\alpha_i - \beta_i & \mbox{for}\, \alpha_i - \beta_i \geq 0  \\
0 & \mbox{for}\, \alpha_i - \beta_i < 0 \end{array} \right.  \quad \mbox{and} \quad {\mu}_i  = \left\{ \begin{array}{cc}
\beta_i - \alpha_i & \mbox{for}\, \beta_i - \alpha_i \geq 0  \\
0 & \mbox{for}\, \beta_i - \alpha_i < 0 \end{array} \right.\]
Note that $\alpha_i - \beta_i = \gamma_i - \mu_i, \gamma_i \geq 0$ and $\mu_i \geq 0$ and hence $$M^*_{\bm{z}^{\bm{\beta}}} M_{\bm{z}^{\bm{\alpha}}} = M^*_{\bm{z}^{\bm{\mu}}} M_{\bm{z}^{\bm{\gamma}}} = M_{\bm{z}^{\bm{\gamma}}} M^*_{\bm{z}^{\bm{\mu}}}.$$
Therefore, $$ \langle M_{\bm{z}^{\bm{\alpha}}} X^*_i f_{\bm{\alpha}}, M_{\bm{z}^{\bm{\beta}}} X^*_i f_{\bm{\beta}} \rangle = \langle M^*_{\bm{z}^{\bm{\beta}}} M_{\bm{z}^{\bm{\alpha}}} X^*_i f_{\bm{\alpha}}, X^*_i f_{\bm{\beta}} \rangle = \langle M^*_{\bm{z}^{\bm{\mu}}} X^*_i f_{\bm{\alpha}}, M^*_{\bm{z}^{\bm{\gamma}}} X^*_i f_{\bm{\beta}} \rangle,$$and, since $$M^*_{\bm{z}^{\bm{\mu}}} X^*_i f_{\bm{\alpha}} = X^*_i M^*_{\bm{z}^{\bm{\alpha}}} f_{\bm{\alpha}},$$ for all $\bm{\alpha}$, we have that $$ \langle M_{\bm{z}^{\bm{\alpha}}} X^*_i f_{\bm{\alpha}}, M_{\bm{z}^{\bm{\beta}}} X^*_i f_{\bm{\beta}} \rangle = \langle X^*_i M^*_{\bm{z}^{\bm{\mu}}} f_{\bm{\alpha}}, X^*_i M^*_{\bm{z}^{\bm{\gamma}}} f_{\bm{\beta}} \rangle = \langle M^*_{\bm{z}^{\bm{\mu}}} f_{\bm{\alpha}}, M^*_{\bm{z}^{\bm{\gamma}}} f_{\bm{\beta}} \rangle.$$Hence $V$ is well-defined and isometric. Moreover, since the span of vectors of the form $$\sum_{|\bm{\alpha}| \leq N} M_{\z^{\bm{\alpha}}} X^*_i f_{\bm{\alpha}}$$ is dense in $H^2(\mathbb{D}^m) \otimes \cle_i$ for $i=1, 2$, by minimality, $V$ is a unitary module map from $H^2 (\mathbb{D}^m) \otimes {\cle_1}$ onto $H^2 (\mathbb{D}^m) \otimes \cle_2$, which concludes the proof. \qed

The above proof will only work if the algebra is generated by functions for which module multiplication  defines doubly commuting isometric operators which happens for the Hardy module on the polydisk. For a more general quasi-free Hilbert module $\clr$, the maps $X^*_i$ identify anti-holomorphic sub-bundles of the bundle $E_{\clr} \otimes \cle_i$, where $E_{\clr}$ is the Hermitian holomorphic line bundle defined by $\clr$. To establish uniqueness, some how one must extend this identification to the full bundles. Equivalently, one has to identify the holomorphic quotient bundles of $E_{\clr} \otimes \cle_1$, and $E_{\clr} \otimes \cle_2$ and must some how lift it to the full bundles. At this point it is not even obvious that the dimensions of $\cle_1$ and $\cle_2$ or the ranks of the bundles are equal. This seems to be an interesting question. Using results on exact sequences of bundles (cf. \cite{GH}), one can establish uniqueness if
$\mbox{dim}\, \cle = \mbox{rank}\, E_{\clh} + 1$.

Another method of defining the isometry $Y : \clh \raro H^2(\mathbb{D}^m) \otimes {\mathcal  E}$, which yields the co-isometric extension of $\clh$, is to set $$Y K(\cdot, \w) \gamma  = \mathbb S_{\mathbb D^m} (\cdot,\w)V_{\w} \gamma, ~ \mbox{for}~ \w \in \Omega, \gamma \in l^2_n.$$ That $Y$ is well defined and isometric follows from the relation of the kernel functions for $\clh$ and $H^2(\mathbb{D}^m) \otimes \mathcal  E$. By uniqueness, then these two constructions of the isometric co-extension must coincide.

\section{An example}

We construct an example of a concrete quasi-free module over the algebra $\mathcal A(\mathbb D^2)$ that illustrates some of the subtlety in dilating to the Hardy module $H^2(\mathbb D^2)$. Consider the submodule ${\mathcal S}:=\{f(z,w) \oplus f(z,z) : f \in H^2(\mathbb{D}^2)\}$ of $H^2(\mathbb{D}^2) \oplus H^2(\mathbb{D})$ over the bi-disk algebra $A(\mathbb{D}^2)$.  The module multiplication on ${\mathcal S}$ is given by the natural action of the algebra $\mathcal A(\mathbb D^2)$ as follows:
$$
(\varphi \cdot f)(z,w) = \varphi(z,w) f(z,w) \oplus \varphi(z,z) f(z,z), \,\varphi \in \mathcal A(\mathbb D^2),\, f\in H^2(\mathbb D^2).
$$
The vector $1 \oplus 1$ generates $\mathcal S$ and the submodule $\mathcal S$ is quasi-free of rank $1$.

Let $\mathcal T$ be a joint $(M^*_{z_1}, M^*_{z_2})$-invariant subspace of the Hardy module $H^2(\mathbb{D}^2)$.
The module action is induced by the two operators $(P_{\mathcal T} M_{z_1}|_{\mathcal T}, P_{\mathcal T} M_{z_2}|_{\mathcal T})$.  Suppose ${\mathcal S}$  is unitarily equivalent to the module $\mathcal T$.  With respect to the orthogonal decomposition, $H^2(\mathbb{D}^2) = \mathcal T \oplus \mathcal T^{\perp}$, we have that
\begin{equation*}
M_{z_1} = \begin{bmatrix}T_1&0\\A_1&N_1\end{bmatrix} \quad \text{and} \quad M_{z_2} = \begin{bmatrix}T_2&0\\A_2&N_2\end{bmatrix}.
\end{equation*}
But
\begin{equation*}
M_{z_1}^* M_{z_1} = \begin{bmatrix}T^*_1&A^*_1\\0&N^*_1\end{bmatrix} \begin{bmatrix}T_1&0\\A_1&N_1\end{bmatrix} = \begin{bmatrix}T^*_1T_1 + A^*_1 A_1&A^*_1 N_1\\N^*_1A_1&N^*_1N_1\end{bmatrix} = \begin{bmatrix}I_{\clt}&0\\0&I_{\clt^{\perp}}\end{bmatrix},
\end{equation*}
and hence, $A_1^* A_1 = 0$ or, equivalently, $A_1 = 0$. Similarly, $A_2 = 0$. Consequently, $\mathcal T$ is a joint $(M_{z_1}, M_{z_2})$-reducing  subspace of $H^2(\mathbb{D}^2)$ which is a contradiction (as none of the reducing subspaces, $H^2(\mathbb{D}) \oplus \{0\}$ or $\{0\} \oplus H^2(\mathbb{D})$, of $H^2(\mathbb{D}^2)$ are unitarily equivalent to ${\mathcal S}$).
Note that $M_{z_1}$ and $M_{z_2}$ are isometries.  Hence, so are $T_1$ and $T_2$.
Therefore, we have proved the following.

\begin{Proposition}
The Hilbert module $\cls$ does not have any resolution $$\cdots \longrightarrow H^2(\mathbb{D}^2) \stackrel{X} \longrightarrow \cls \longrightarrow 0$$with $X$ a co-isometric module map.
\end{Proposition}

There is a useful alternative description of the Hilbert module ${\mathcal S}$ discussed above based on Proposition \ref{main}.  First, we observe that the linear subspace $\{f(z,z):f\in H^2(\mathbb D^2)\} \subseteq \clo(\bigtriangleup)$, the space of holomorphic functions on $\bigtriangleup$, is not isomorphic to the Hardy module $H^2(\mathbb D)$ but rather to the Bergman module $L^2_a(\mathbb{D})$. Let $H_{1/2}(\mathbb D^2)$ be the Hilbert space of holomorphic functions on
the bi-disc $\mathbb D^2$ determined by the positive definite kernel
$$B_{1/2}(\z,\w) = \frac{1}{(1-z_1\bar{w}_1)^{\frac{1}{2}}} \,  \frac{1}{(1-z_2\bar{w}_2)^{\frac{1}{2}}},\,\z=(z_1,z_2),\,\w=(w_1,w_2) \in \mathbb D^2.$$
We recall that the restriction map ${\rm res}: H_{1/2}(\mathbb D^2) \to H^2(\mathbb D)$ defined by the formula $f \mapsto f_{|{\rm res~}\bigtriangleup}$ (that is, restriction to the diagonal $\bigtriangleup$), is a co-isometry.  The orthocompliment $\mathcal Q$ of the kernel of this map in  $H_{1/2}(\mathbb D^2)$, considered as a quotient module, is therefore isometrically isomorphic to the Hardy module $H^2(\mathbb D)$.  Let $K_\mathcal Q$ denote the reproducing kernel for the module $\mathcal Q$. Then
$$K_{\mathcal Q}(\z, \w) = K_{H_{1/2}(\mathbb{D}^2)}(\z, \w) - (z_1 - z_2) \cdot \chi (\z, \w) \cdot (\bar{w_1} - \bar{w_2}),$$
for some positive definite kernel $\chi$ on the bi-disk $\mathbb{D}^2$. By a result of Aronszajn \cite{aron}, the kernel function for $\mathcal S$ is given by
$$K_{\mathcal S}(\z, \w) = K_{H^2(\mathbb{D}^2)} (\z, \w) + K_{H_{1/2}(\mathbb{D}^2)}(\z, \w) - (z_1 - z_2) \cdot \chi (\z, \w)\cdot (\bar{w_1} - \bar{w_2}).$$
This fact requires an identification of the space associated with the sum of two kernel functions and the space $\cls$ constructed above.

Now let us consider the question of whether $\mathcal S$ possesses an isometric co-extension to $H^2_{\mathcal  E}(\mathbb{D}^2)$ for some Hilbert space $\mathcal  E$. By Theorem \ref{cor1}, this is equivalent to the positive definiteness of
$$H(\z, \w) =  (1 - z_1 \bar{w}_1) (1 - z_2 \bar{w}_2) K_{\mathcal S}(\z, \w).$$
However, this is not the case since the restriction of $H(\z, \w)$ to the diagonal $\bigtriangleup$ is not positive. More precisely,
$$H(z,z;w,w) = (1 - z \bar{w})^2\{(1 - z \bar{w})^{-2} + (1 - z \bar{w})^{-1}\} = 1 + (1 - z \bar{w}) = 2 - z \bar{w};$$ which is not positive definite. Therefore, ${\mathcal S}$ is a contractive quasi-free Hilbert module over the bi-disk algebra but the kernel function $K_{\mathcal S}$ does not admit the S\"{z}ego kernel as a factor. Thus this provides another proof that the module $\mathcal S$ does not possess a $H^2(\mathbb D^2)$ co-isometric extension.

\section{Other Kernel functions in Several Variables}

The results of Section 3 apply to more than the case of the Hardy module which we stated in Theorem \ref{dil}. More precisely, we have

\begin{Theorem}\label{allkernel} Let $\clm$ be a Hilbert module over $\cla(\Omega)$ for $\Omega \subseteq \mathbb{C}^m$ (or $\mathbb{C}[\z]$ for $\Omega = {\mathrm B}^m$ or $\mathbb{D}^m$) having the kernel function $K(\z, \w) = k(\z, \w) I_{\mathcal  E}$, where $k(\z, \w)$ is a scalar kernel function and $\mathcal  E$ is a Hilbert space. Let $\clh$ be a Hilbert module over the same algebra with kernel function $K_{\clh}(\z, \w)$ such that $k^{-1}(\bm{M}, \bm{M}^*) \geq 0$, where $\bm{M}$ is the coordinate multipliers on $\clh$. Then $\clh$ can be realized as a quotient module of $\clm \otimes \mathcal Q$ over the same algebra for some Hilbert space $\mathcal Q$, and conversely. Hence $\clh$ has an isometric co-extension to $\clm \otimes \mathcal Q$ for some Hilbert space $\mathcal Q$ if and only if $k^{-1}(\bm{M}, \bm{M}^*) \geq 0$ if and only if $k^{-1} K_{\clh} \geq 0$.
\end{Theorem}

We note that the operator positivity assumption in the above theorem includes implicitly the additional hypothesis that one can define a functional calculus so that $k^{-1} (\bm{M}, \bm{M}^*)$ makes sense for the kernel function $k$. It was pointed out in the paper by Arazy and Englis \cite{AE} that for many reproducing kernel Hilbert spaces, one can define such a $\frac{1}{k}$-calculus. Thus our result is true to that extent. We now provide some examples to which these results apply.

The kernel function for the Bergman module, $L^2_a(\mathbb{D})$, is ${\mathrm B}_1(z, w) = (1-z \bar{w})^{-2}$. Therefore, the kernel function for $L^2_a(\mathbb{D}^m)$ is the product $${\mathrm B}_m (\z, \w) = \Pi_{i=1}^{m} (1 - z_i \bar{w}_i)^{-2}, \quad  \z, \w \in \mathbb{D}^m.$$
Applying Theorem \ref{allkernel} we obtain the following result

\begin{Corollary}
If $\clh$ is a contractive reproducing kernel Hilbert module over $\cla(\mathbb{D}^m)$, then $\clh$ has an $L^2_a(\mathbb{D}^m) \otimes \cle$ isometric co-extension if and only if $${\mathrm B}_{\mathbb{D}^m}^{-1}(\bm{M}, \bm{M}^*) \geq 0,$$or, equivalently, if and only if $${\mathrm B}_{\mathbb{D}^m}^{-1} K \geq 0,$$ where $K$ is the kernel function for $\clh$.
\end{Corollary}

Note that if ${\mathrm B}_{\mathbb{D}^m}^{-1} K \geq 0$, it follows that $\mathbb{S}_{\mathbb{D}^m}^{-1} K \geq 0$. Hence, if $\clh$ has an isometric Bergman space co-extension, it also has a Hardy space one, a result which can be proved directly.

Further, note that one could state similar results for the weighted Bergman spaces on $\mathbb{D}^m$. We omit the details.

We now consider examples on the unit ball. Let $\B^m:=\{\z\in \C^m: |z_1|^2 +\cdots + |z_m|^2 <1 \}$ be the Euclidean unit ball
and $K$ be a positive definite kernel on $\B^m$.  Let
$T_\ell$, $1\leq \ell \leq m$, be the operator defined on the
normed linear space $\mathcal H^0 = \bigvee \{K(\cdot, \w): \w \in \mathbb B^m\}$ by the formula $T_\ell K(\cdot, \w) = \bar{w}_\ell K(\cdot,\w)$.
The following extension of Lemma \ref{bdd} gives a criterion for the contractivity of the operator $\sum_{\ell=1}^m T_\ell^* T_\ell \leq I$.
\begin{Corollary}\label{DAkernel}
Let $K$ be a positive definite kernel function on the unit ball
$\B^m$. The commuting $m$-tuple $\mathbf T=(T_1, \ldots , T_m)$ of linear maps on
$\mathcal H^0 \subseteq \mathcal H_K$ satisfies the contractivity condition
$\sum_{\ell=1}^m T_\ell^* T_\ell \leq I$ if and only if the function
$(1 - \inner{\z}{\w})K(\z,\w)$ is positive definite.
\end{Corollary}
\begin{proof}
We note that $\sum_{\ell=1}^m T_\ell^* T_\ell$ is a Hermitian operator.
Therefore, it is enough to compute
$$\sup\{\inner{\sum_{\ell=1}^m T_\ell^* T_\ell \eta}{\eta}:
\eta = \sum_{i=1}^n K(\cdot, \w_i) x_i\},\, \w_i \in \mathbb B^m,\, n \in \mathbb N.$$
But the computation for each term of the summand, is the same as the one in the proof of Lemma \ref{bdd}.  Adding all of these inequalities completes the proof.
\end{proof}

Suppose $\mathcal M$ is a Hilbert module over the ball algebra $\mathcal A(\mathbb B^m)$ and let $K_\mathcal M$ be its reproducing kernel. The operators $M_\ell$ of multiplication by the coordinate functions on $\mathcal M$ satisfy the inequality
$\sum_{\ell=1}^m M_\ell^* M_\ell \leq I$ if and only if $K_\mathcal M (\z,\w) = (1 - \inner{\z}{\w})^{-1}K(\z,\w)$ for some positive definite kernel $K$ on the ball $\mathbb B^m$.
The Hilbert module over the ball algebra $\mathcal A(\mathbb B^m)$ corresponding to the kernel $(1-\inner{\z}{{\w}})^{-1}$, $\z,\w \in \mathbb B^m$, is the Drury-Arveson space $H^2_m$. For this module, the operator inequality of the lemma is evident.  Let $\mathcal N$ be the Hilbert space corresponding to the positive definite kernel $K$ which appears in the factorization of $K_\mathcal M$.
Now, assume that $\sum_{\ell=1}^m M_\ell M_\ell^* \leq I$.
It then follows from  an extension of Theorem \ref{allkernel} that the Hilbert module $\mathcal M$ admits an isometric co-extension to the Drury-Arveson space  $H^2_m \otimes \mathcal N \equiv H^2_m(\mathcal N)$. Thus we have obtained a special case of Arveson's dilation result (cf. \cite{Ar}):
\begin{Proposition}
Let $\mathcal M$ be a quasi-free Hilbert module over the ball algebra $\mathcal A(\mathbb B^m)$. Suppose that the $m$-tuple of operators defined to be multiplication by the coordinate functions on $\mathcal M$ satisfies the operator inequality $\sum_{\ell=1}^m M_\ell M_\ell^* \leq I$.  Then $\mathcal M$ can be realized as a quotient module of $H^2_m \otimes \mathcal N$
for some Hilbert space $\mathcal N$.
\end{Proposition}

Just as we did for the Bergman module for the polydisk, we can also consider possible dilations on the ball under the assumption that $(1 - \langle \z, \w \rangle)^k K$ is positive definite. For $k=1$ we have the previous result. For $k=m$, we obtain the result for the Hardy module over the ball and for $k=m+1$, we obtain the result for the Bergman module on the ball. Again, the existence of a dilation for one value of $k$ implies the existence for all smaller $k$, and hence always for the Drury-Arveson space if for any $k$.

A final observation concerns fractional exponents for which one obtain Besov-like spaces. Arguments such as those given in this section to yield additional relationships between these Hilbert modules and should be worth considering.

\section{Curvature Inequality}

Let $\mathcal R(k)$ be a scalar reproducing kernel Hilbert module for the n.n.d. kernel function $k$ over the polynomial ring $\mathbb C[\z]$ consisting of holomorphic functions on some bounded domain $\Omega$ in $\mathbb C^m$.  Assume that  $\mathbb C[\z] \subseteq \mathcal R(k)$ is dense in $\mathcal R(k)$.  It then follows (\cite{CD}) that the map $\w \mapsto \cap_{i=1}^m\ker (M_i - w_i)^*$ is anti-holomorphic from $\Omega$ to the projective space of $\clr(k)$ and defines an anti-holomorphic line bundle $E^*_{\mathcal R(k)}$ on  $\Omega$ via the map $\w \mapsto k(\cdot, \w)$, $\w \in \Omega$. The dual bundle $E_{\clr(k)}$ is a holomorphic bundle over $\Omega$. Moreover, $E_{\mathcal R(k)}$ is Hermitian, where the Hermitian metric on the line bundle is induced by the standard Hermitian inner product on $\mathcal R(k)$. In other words, with respect to the frame $\{s\}$, where $s(\bm{w}) = k(\cdot, \w)$, the Hermitian structure of $E_{\mathcal R(k)}$ defines a Hermitian form
  $$h (\w) = <s(\w), s(\w)> = \|k(\cdot, \w)\|^2, \quad \quad \w \in \Omega.$$ Then the canonical Chern connection on $E_{\mathcal R(k)}$ is given by $$\triangledown = \partial h. h^{-1},$$ with the curvature form

\begin{equation} \label{curvE} \mathcal K_{\mathcal R(k)}(\w) = - \frac{1}{2} \partial \bar{\partial} \mbox{log}\, h = - \frac{1}{2} \sum_{i,j=1}^m \frac{\partial}{\partial w_i} \frac{\partial}{\partial \bar{w}_j} \log k(\w,\w) \,dw_i \wedge d\bar{w}_j, \,\, \w \in \Omega.\end{equation}

We want to compare the curvatures of the bundles associated with a quotient module and the bundle for the isometric co-extension. First, we need to recall some results from complex geometry concerning curvatures of sub-bundles and quotient bundles (cf. \cite{GH}, pp. 78-79).

Let $E$ be a Hermitian holomorphic bundle over $\Omega \subseteq \mathbb{C}^m$ (possibly infinite rank) and $F$ be a holomorphic sub-bundle of $E$ such that the quotient $Q = E/F$ is also holomorphic. Let $\triangledown_E$ denote the Chern connection on $E$ and $\Theta_E$ the corresponding curvature form. There are two canonical connections that we can define on $F$ and the quotient bundle $Q$. The first ones are the Chern connections $\triangledown_F$ and $\triangledown_Q$ on $F$ and $Q$, respectively. To obtain the second connections, let $P$ denote the projection-valued bundle map of $E$ so that $P(\z)$ is the orthogonal projection of $E(\z)$ onto $F(\z)$. Then $$\triangledown_{PE} = P \triangledown_E P \quad \quad \mbox{and}\quad \quad \triangledown_{P^{\perp}E} = P^{\perp} \triangledown_E P^{\perp},$$ define connections on $F$ and $Q$, respectively, where $P^{\perp} = I - P$ and $Q$ is identified fiber wise with $P^{\perp} E$. The following result from complex geometry  relates the curvatures for these pairs of connections.

\begin{Theorem}
If $F$ is a holomorphic sub-bundle of the holomorphic bundle $E$ over $\Omega \subseteq \mathbb{C}^m$ such that $E/F$ is holomorphic, then the curvature functions for the connections $\triangledown_F, \, \triangledown_{PE}, \, \triangledown_Q$ and $\triangledown_{P^{\perp}E}$ satisfy $$\Theta_F (\w) \geq \Theta_{PE}(\w) \quad \quad \mbox{and}\quad \quad \Theta_Q(\w) \leq \Theta_{P^{\perp} E}(\w), \quad \quad \w \in \Omega.$$
\end{Theorem}

The proof is essentially a matrix calculation involving the off-diagonal entries of $\triangledown_E$, one of which is the second fundamental form and the other its dual (cf. \cite{GH}). (Note in \cite{GH}, $E$ is finite rank but the proof extends to the more general case.)

We apply this result to Hilbert modules as follows.

\begin{Theorem}
Let $\clr$ be a quasi-free Hilbert module over $A(\Omega)$ for $\Omega \subseteq \mathbb{C}^m$ (or over $\mathbb{C}[\z]$) of multiplicity one and $\cls$ be a submodule of $\clr \otimes \cle$ for a Hilbert space $\cle$ such that the quotient module $\mathcal Q = (\clr \otimes \cle)/\cls$ is in $B_n(\Omega)$ for some $1 \leq n < \infty$. If $E_{\clr}$ and $E_{\mathcal Q}$ are the corresponding Hermitian holomorphic bundles over $\Omega$, then $$P^{\perp}(\w) ( \Theta_{E_{\clr}}(\w) \otimes I_{\cle}) P^{\perp}(\w) \geq \Theta_{\mathcal Q}(w), \quad \w \in \Omega.$$
\end{Theorem}
\NI\textsf{Proof.} The result follows from the previous theorem by setting $E = E_{\clr} \otimes \cle, \, F = E_\cls$ and $Q = E_\mathcal Q$. \qed

In particular, we have the following extremal property of the curvature functions.

\begin{Theorem}
A necessary condition for a Hilbert module $\clh$ in $B_n(\Omega)$ over $A(\Omega)$, $\Omega \subseteq \mathbb{C}^m$, to have $\clr \otimes \cle$ for some Hilbert space $\cle$ as an isometric co-extension is that $$ \Theta_{E_{\clr}}(\w) \otimes I_{\cle} \geq \Theta_{\clh}(w), \quad \w \in \Omega.$$
\end{Theorem}

The converse of this result is false, but it is likely valid for some stronger notion of positivity. We plan to take up this matter in the future.

Recalling Corollary \ref{DAkernel}, we see that any contractive reproducing kernel Hilbert module $\clm(K)$ over the polynomial algebra $\mathbb{C}[\z]$ satisfies the inequality $\mathbf K_{H^2_m} - \mathbf K_{\clm(K)} \geq 0$. This is a generalization of the curvature inequality for the disk from \cite{GM}, see also \cite{U, GM-NSN}.

The above inequality shows in view of Corollary \ref{main} that the module $H^2_m$ is an extremal element in the set of contractive Hilbert modules over the algebra $\mathbb{C}[\z]$.  Similarly, for the polydisk $\mathbb D^m$, the Hardy module is an extremal element in the set of those modules over the algebra $\mathcal A(\mathbb D^m)$ which admit a co-extension to the Hardy space $H^2(\mathbb D^m) \otimes \mathcal  E$.

\end{document}